\documentclass{amsart}

\usepackage{hyperref}

\swapnumbers

\numberwithin{equation}{section}

\newtheorem{Thm}{Theorem}[section]
\newtheorem{Prop}[Thm]{Proposition}
\newtheorem{Lem}[Thm]{Lemma}
\theoremstyle{definition}
\newtheorem{Expl}[Thm]{Example}

\newcommand{\R}{\mathbb{R}}

\newcommand{\E}[1]{{\rm E}(#1)}
\newcommand{\e}{{\rm e}}
\newcommand{\del}{\delta}
\newcommand{\diam}{\operatorname{diam}}
\newcommand{\dis}{\operatorname{dis}}
\newcommand{\eps}{\varepsilon}
\newcommand{\es}{\emptyset}
\newcommand{\GH}{\text{\rm GH}}
\renewcommand{\H}{\text{\rm H}}
\renewcommand{\rho}{\varrho}
\newcommand{\Sig}{\Sigma}
\newcommand{\sm}{\setminus}
\newcommand{\sub}{\subset}
\newcommand{\defl}{\mathrel{\mathop:}=}

\title{Metric stability of trees and tight spans}

\author{Urs Lang}
\address{Department of Mathematics, ETH Zurich, 8092 Zurich, Switzerland}
\email{urs.lang@math.ethz.ch}

\author{Ma\"el Pav\' on}
\address{Department of Mathematics, ETH Zurich, 8092 Zurich, Switzerland}
\email{mael.pavon@math.ethz.ch}

\author{Roger Z\"{u}st}
\address{D\'{e}partement de math\'{e}matiques, Universit\'{e} de Fribourg, 1700 Fribourg, Switzerland}
\email{roger.zuest@unifr.ch}

\thanks{Research supported by the Swiss National Science Foundation.}

\date{March 27, 2013}

\begin{document}

\begin{abstract}
We prove optimal extension results for roughly isometric relations 
between metric ($\R$-)trees and injective metric spaces. This yields sharp 
stability estimates, in terms of the Gromov--Hausdorff (GH) distance, 
for certain metric spanning constructions: The GH distance of two
metric trees spanned by some subsets is smaller than or equal to 
the GH~distance of these sets. The GH~distance of the injective hulls,
or tight spans, of two metric spaces is at most twice the GH~distance between 
themselves.
\end{abstract}

\maketitle


\section{Introduction}

The main purpose of this note is to provide an optimal stability result,
in terms of the Gromov--Hausdorff distance, for Isbell's~\cite{Isb}
injective hull construction $X \mapsto \E X$ for metric spaces. 
Roughly speaking, $\E X$ is a smallest injective metric space containing 
an isometric copy of~$X$ (all relevant definitions will be reviewed later 
in this paper). Here, a metric space $Y$ is called {\em injective} 
if for any isometric embedding $i \colon A \to B$ of metric spaces and any 
$1$-Lipschitz (i.e., distance-nonincreasing) map $f \colon A \to Y$ there 
exists a $1$-Lipschitz extension $g \colon B \to Y$ of $f$, so that
$g \circ i = f$ (see \cite[Section~9]{AdaHS} for the general categorical 
notion). Examples of injective metric spaces include the real line $\R$, 
$l_\infty(I)$ for any index set $I$, and all complete metric trees;
however, by Isbell's result, this list is by far not exhaustive. 
Injective metric spaces are complete, geodesic, and contractible 
and share a number of remarkable properties. We refer 
to~\cite[Sections~2 and~3]{L} for a recent survey of injective metric spaces 
and hulls. 

An alternative, but equivalent, description of $\E X$ was given later by 
Dress~\cite{Dre}, who called it the tight span of $X$.
If $X$ is compact, then so is $\E X$, and if $X$ is finite, $\E X$ has 
the structure of a finite polyhedral complex of dimension at most $|X|/2$
with cells isometric to polytopes in some finite-dimensional $l_\infty$~space. 
If every quadruple of points in $X$ admits an isometric embedding into some 
metric tree, then so does $X$ itself, and $\E X$ provides the minimal 
complete such tree. This last property makes the 
injective hull/tight span construction a useful tool in phylogenetic 
analysis. Based on genomic differences an evolutionary distance 
between similar species is defined, and the construction
may then be applied to this finite metric space. Due to noise in 
the measurements or systematic errors, the process will rarely yield a tree, 
but (the $1$-skeleton of) the resulting polyhedral complex may still give 
a good indication on the phylogenetic tree one tries to reconstruct 
(compare~\cite{DreHM, DreMT} and the references there).

In view of these applications, and also from a purely geometric 
perspective, it is interesting to know how strongly the injective hull is
affected by small changes of the underlying metric space.
The dissimilarity of two metric spaces $A,B$ 
is conveniently measured by their Gromov--Hausdorff distance $d_\GH(A,B)$. 
Moezzi~\cite[Theorem~1.55]{Moe} observed that $d_{\GH}(\E A,\E B)$ is not 
larger than eight times $d_{\GH}(A,B)$. Here it is now shown that in fact
\[ 
d_{\GH}(\E A,\E B) \le 2\,d_{\GH}(A,B),
\]
and an example is constructed to demonstrate that the factor two is optimal
(see Section~3). Furthermore, we prove that if both $\E A$ and $\E B$ are 
metric trees (in the most general sense of $\R$-trees), then
\[ 
d_{\GH}(\E A,\E B) \le d_{\GH}(A,B), 
\]
without a factor two. In particular, this implies that if $X,Y$ are two 
finite simplicial metric trees with sets of terminal vertices $A,B$,
respectively, then $d_{\GH}(X,Y) \le d_{\GH}(A,B)$. This result (which we have 
not been able to find in the literature) is not as obvious 
as it may appear at first glance. A complication arises from the fact that 
for the respective vertex sets $V_X,V_Y$, 
it is not true in general that $d_{\GH}(V_X,V_Y) \le d_{\GH}(A,B)$,
not even for combinatorially equivalent binary trees.
For instance, consider the two trees $X,Y$ depicted below, with the indicated 
edge lengths. 
\bigskip\newline
\setlength{\unitlength}{6mm}
\thicklines
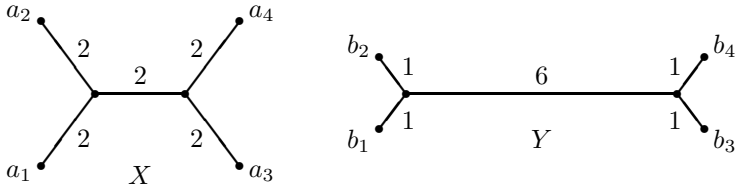
\hspace*{\fill}
\begin{picture}(4.4,4)(-2.2,-2)
\put(-1,0){\line(1,0){2}}\put(-1,0){\circle*{.2}}\put(1,0){\circle*{.2}}
\put(1,0){\line(3,4){1.2}}\put(2.2,1.6){\circle*{.2}}
\put(1,0){\line(3,-4){1.2}}\put(2.2,-1.6){\circle*{.2}}
\put(-1,0){\line(-3,4){1.2}}\put(-2.2,1.6){\circle*{.2}}
\put(-1,0){\line(-3,-4){1.2}}\put(-2.2,-1.6){\circle*{.2}}
\put(0,0.2){\makebox(0,0)[b]{2}}
\put(1.4,0.8){\makebox(0,0)[br]{2}}
\put(1.4,-0.8){\makebox(0,0)[tr]{2}}
\put(-1.4,0.8){\makebox(0,0)[bl]{2}}
\put(-1.4,-0.8){\makebox(0,0)[tl]{2}}
\put(2.4,1.6){\makebox(0,0)[bl]{$a_4$}}
\put(2.4,-1.6){\makebox(0,0)[tl]{$a_3$}}
\put(-2.4,1.6){\makebox(0,0)[br]{$a_2$}}
\put(-2.4,-1.6){\makebox(0,0)[tr]{$a_1$}}
\put(0,-1.6){\makebox(0,0)[t]{$X$}}
\end{picture}
\hfill
\begin{picture}(7.2,4)(-3.6,-2)
\put(-3,0){\line(1,0){6}}\put(-3,0){\circle*{.2}}\put(3,0){\circle*{.2}}
\put(3,0){\line(3,4){0.6}}\put(3.6,0.8){\circle*{.2}}
\put(3,0){\line(3,-4){0.6}}\put(3.6,-0.8){\circle*{.2}}
\put(-3,0){\line(-3,4){0.6}}\put(-3.6,0.8){\circle*{.2}}
\put(-3,0){\line(-3,-4){0.6}}\put(-3.6,-0.8){\circle*{.2}}
\put(0,0.2){\makebox(0,0)[b]{6}}
\put(3.1,0.4){\makebox(0,0)[br]{1}}
\put(3.1,-0.4){\makebox(0,0)[tr]{1}}
\put(-3.1,0.4){\makebox(0,0)[bl]{1}}
\put(-3.1,-0.4){\makebox(0,0)[tl]{1}}
\put(3.8,0.8){\makebox(0,0)[bl]{$b_4$}}
\put(3.8,-0.8){\makebox(0,0)[tl]{$b_3$}}
\put(-3.8,0.8){\makebox(0,0)[br]{$b_2$}}
\put(-3.8,-0.8){\makebox(0,0)[tr]{$b_1$}}
\put(0,-0.8){\makebox(0,0)[t]{$Y$}}
\end{picture}
\hspace*{\fill}
\bigskip\newline
The correspondence between $A \defl \{a_1,\dots,a_4\}$ and 
$B \defl \{b_1,\dots,b_4\}$ that relates $a_i$ to $b_i$ distorts all 
distances by an additive error of two. Since the diameters of $A$ and $B$ 
also differ by two, no correspondence (i.e., left- and right-total relation)
between $A$ and $B$ has (maximal) distortion less than two. The 
Gromov--Hausdorff distance equals one half this minimal number (see Section~3), 
so $d_\GH(A,B) = 1$. Similar considerations show that $d_\GH(V_X,V_Y) = 2$.
Yet, $d_\GH(X,Y) = 1$. For the proof, points in $X$ and $Y$ need to be related 
in a non-canonical way. 


\section{Extension of roughly isometric relations}

As just indicated, the Gromov--Hausdorff distance may be characterized 
in terms of the additive distortion of relations between the two given 
metric spaces. Therefore, in this section, we begin by studying the 
possibility of extending relations without increasing the distortion.

Let $X,Y$ be two metric spaces. We write $|xx'|$ for the distance 
of two points $x,x' \in X$ and, likewise, $|yy'|$ for the distance of
$y,y' \in Y$.
Given a relation $R$ between $X$ and $Y$, i.e., a subset of $X \times Y$,
the {\em distortion} of $R$ is defined as the (possibly infinite) number
\[
\dis(R) \defl \sup\bigl\{ \bigl| |xx'| - |yy'| \bigr| : 
(x,y), (x',y') \in R \bigr\}.
\]
In case $R$ is given by a map $f \colon X \to Y$, we write $\dis(f)$ 
for $\dis(R)$. If $\dis(f) \le \eps$ for some $\eps \ge 0$, then $f$ is 
called {\em $\eps$-roughly isometric}. This means that 
\[
|xx'| - \eps \le |f(x)f(x')| \le |xx'| + \eps
\]
for every pair of points $x,x' \in X$. See~\cite[Chapter~7]{BurBI} 
and~\cite[Chapter~7]{BuyS} for this terminology. We denote by 
$\pi_X \colon X \times Y \to X$ and $\pi_Y \colon X \times Y \to Y$ 
the canonical projections. For a set $A \sub X$, we say that 
$A$ \emph{spans}~$X$ if, for every pair $(x,x') \in X \times X$, 
\[
|xx'| = \sup_{a \in A} \bigl( |xa| - |x'a| \bigr);
\]
equivalently, for all $\eps > 0$ there is an $a_\eps \in A$ such that 
$|xx'| + |x'a_\eps| \le |xa_\eps| + \eps$. The definition is motivated by the 
fact that the injective hull of a metric space $A$ may be characterized as 
an injective metric extension $X \supset A$ spanned by $A$, 
see Proposition~\ref{Prop:inj-hull-spanned} below.
For a constant $\alpha \ge 0$, a set $S \sub X$ is 
called an {\em $\alpha$-net} in $X$ if for every $x \in X$ there exists a 
$z \in S$ such that $|xz| \le \alpha$.

\begin{Prop}\label{Prop:injective-ext}
Suppose that $X,Y$ are two injective metric spaces.
If $R \sub X \times Y$ is a set with $\alpha \defl \dis(R)/2 < \infty$ and
the property that $\pi_X(R)$ spans $X$, there exists an extension 
$R \sub \bar R \sub X \times Y$ such that $\pi_X(\bar R)$ is an $\alpha$-net 
in $X$ and $\dis(\bar R) = \dis(R)$.
\end{Prop}

In particular, every $\eps$-roughly isometric map $f \colon A \to Y$ 
defined on a set $A \sub X$ that spans $X$ admits an $\eps$-roughly
isometric extension $\bar f \colon S \to Y$ to some $\eps/2$-net 
$S$ in $X$ and, hence, also a $2\eps$-roughly isometric extension
$\hat f \colon X \to Y$. Below we shall use the simple fact 
that every injective metric space $Y$ is {\em hyperconvex}~\cite{AroP} 
(the converse is true as well). This means that for every family 
$\{(y_i,r_i)\}_{i \in I}$ in $Y \times \R$ with the property that 
$r_i + r_j \ge |y_i y_j|$ for all pairs of indices $i,j \in I$, 
there is a point $y \in Y$ such that $|y y_i| \le r_i$ for all $i \in I$.

\begin{proof}
It suffices to show that for every set $R \sub X \times Y$ with 
$\alpha \defl \dis(R)/2 < \infty$ and the property that $\pi_X(R)$ 
spans $X$ and for every $\bar x \in X$ there exists a pair 
$(x_0,y_0) \in X \times Y$ such that $|\bar x x_0| \le \alpha$ and
\[
\dis\bigl( R \cup \{(x_0,y_0)\} \bigr) = \dis(R).
\] 
The general result then follows by an application of Zorn's lemma.

Let such $R$ and $\bar x$ be given, and put $\alpha \defl \dis(R)/2$. 
For all $(x,y),(x',y') \in R$,
\[
\bigl| |x x'| - |y y'| \bigr| \le 2\alpha
\]
and $(|x \bar x| + \alpha) + (|x' \bar x| + \alpha) \ge |x x'| + 2\alpha 
\ge |y y'|$. Hence, since $Y$ is hyperconvex, there is a 
point $y_0 \in Y$ such that for all $(x,y) \in R$,
\[
|y y_0| \le |x \bar x| + \alpha.
\]
Furthermore, since $\pi_X(R)$ spans $X$, for every $(x,y) \in R$ and 
$\eps > 0$ there exists $(x_\eps,y_\eps) \in R$ such that 
$|x \bar x| + |\bar x x_\eps| \le |x x_\eps| + \eps$ and, hence,   
\[
|y y_0| \ge |y y_\eps| - |y_0 y_\eps|
\ge (|x x_\eps| - 2\alpha) - (|\bar x x_\eps| + \alpha)
\ge |x \bar x| - 3 \alpha - \eps.
\]
Since this holds for all $\eps > 0$, it follows that 
$|y y_0| \ge |x \bar x| - 3 \alpha$.
For every $(x,y) \in R$, put $r(x,y) \defl |y y_0| + 2\alpha$, and set
$r(\bar x) \defl \alpha$. 
We have $r(x,y) + r(\bar x) = |y y_0| + 3\alpha \ge |x \bar x|$
and $r(x,y) + r(x',y') \ge |y y'| + 4\alpha \ge |x x'| + 2\alpha \ge |x x'|$,
for all $(x,y),(x',y') \in R$.
Thus, since $X$ is hyperconvex, there exists a point 
$x_0 \in X$ such that  
\[
|x x_0| \le r(x,y) = |y y_0| + 2\alpha
\]
and $|\bar x x_0| \le r(\bar x) = \alpha$ for all $(x,y) \in R$. Then also 
\[
|y y_0| \le |x \bar x| + \alpha \le |x x_0| + |\bar x x_0| + \alpha
\le |x x_0| + 2\alpha
\]
and so $\bigl| |x x_0| - |y y_0| \bigr| \le 2\alpha = \dis(R)$ for all 
$(x,y) \in R$.
\end{proof}

Now we focus on trees. A metric space $X$ is called {\em geodesic}
if for every pair of points $x,x' \in X$ there is a geodesic segment
$xx' \sub X$ connecting the two points, i.e., the image of an isometric 
embedding of the interval $[0,|xx'|]$ that sends $0$ to $x$ 
and $|xx'|$ to $x'$. By a {\em metric tree} $X$ we mean a geodesic metric 
space with the property that for any triple $(x,y,z)$ of points in $X$ and 
any geodesic segments $xy,xz,yz$ connecting them, $xy \sub xz \cup yz$.
Thus, geodesic triangles in $X$ are isometric to tripods, and geodesic 
segments are uniquely determined by their endpoints. For the next result
we need to sharpen the above assumption that $\pi_X(R)$ spans $X$. We say 
that a subset $A$ of a metric space $X$ {\em strictly spans}~$X$ if for 
every pair $(x,x') \in X \times X$ there exists an $a \in A$ such that 
$|xx'| + |x'a| = |xa|$.

\begin{Prop} \label{Prop:trees-ext}
Suppose that $X$ is a metric tree and $Y$ is an injective metric 
space. If $R \sub X \times Y$ is a set with the property that $\pi_X(R)$ 
strictly spans $X$, there exists an extension $R \sub \bar R \sub X \times Y$ 
such that $\pi_X(\bar R) = X$ and $\dis(\bar R) = \dis(R)$. 
\end{Prop}

In particular, every $\eps$-roughly isometric map $f \colon A \to Y$ 
defined on a set $A \sub X$ that strictly spans $X$ admits an $\eps$-roughly 
isometric extension $\bar f \colon X \to Y$. 

\begin{proof}
It suffices to show that for every set $R \sub X \times Y$ with 
$\dis(R) < \infty$ and the property that $\pi_X(R)$ spans $X$ and for every 
$\bar x \in X$ there exists a point $\bar y \in Y$ such that 
\[
\dis\bigl( R \cup \{(\bar x,\bar y)\} \bigr) = \dis(R).
\] 
As above, the general result then follows by an application of Zorn's lemma.

Thus let such $R$ and $\bar x$ be given. Put $\alpha \defl \dis(R)/2$. 
As in the proof of Proposition~\ref{Prop:injective-ext}, there exists a point 
$y_0 \in Y$ with the property that
\[
|y y_0| \le |x \bar x| + \alpha
\]
for all $(x,y) \in R$. 
Let $S$ be the set of all $(x,y) \in R$ with $|y y_0| < |x \bar x| - \alpha$.
If $S = \es$, then $\bigl| |x \bar x| - |y y_0| \bigr| \le \alpha \le \dis(R)$ 
for all $(x,y) \in R$; in particular, $\bar y \defl y_0$ has the desired 
property.
Suppose now that $S \ne \es$, and fix an arbitrary $(x_1,y_1) \in S$.
Since $\pi_X(R)$ strictly spans $X$, there exists a pair $(x_2,y_2) \in R$
such that $|x_1 \bar x| + |\bar x x_2| = |x_1 x_2|$. 
Now choose $\bar y \in Y$ so that $|\bar y y_0| \le \alpha$ and 
$|\bar y y_2| \le |y_0 y_2| - \alpha$. 
Note that $|y_0 y_2| \le |\bar x x_2| + \alpha$, so 
$|\bar y y_2| \le |\bar x x_2|$. For all $(x,y) \in R$,
\[
|y \bar y| \le |y y_0| + |\bar y y_0| \le |y y_0| + \alpha 
\le |x \bar x| + 2 \alpha.
\]
To estimate $|\bar y y|$ from below, note first that if $(x,y) \in R \sm S$, 
then 
\[
|y \bar y| \ge |y y_0| - |\bar y y_0| \ge |y y_0| - \alpha 
\ge |x \bar x| - 2 \alpha.
\]
Secondly, let $(x,y) \in S$. 
Consider the tripod $xx_1 \cup xx_2 \cup x_1x_2$, and note that 
$\bar x \in x_1x_2$. Since $(x,y),(x_1,y_1) \in S$, the strict inequality 
\[
|x x_1| \le |y y_1| + 2\alpha \le |y y_0| + |y_0 y_1| + 2\alpha 
< |x \bar x| + |\bar x x_1|
\]
holds, so $\bar x \not\in xx_1$ and therefore $\bar x \in xx_2$.
We conclude that
\[
|y \bar y| \ge |y y_2| - |\bar y y_2| \ge (|x x_2| - 2 \alpha) - |\bar x x_2| 
= |x \bar x| - 2 \alpha.
\]
This shows that $\bigl| |x \bar x| - |y \bar y| \bigr| \le 2\alpha = \dis(R)$ 
for all $(x,y) \in R$.
\end{proof}

The following example shows that Proposition~\ref{Prop:trees-ext} 
is no longer true in general if the word ``strictly'' is omitted.

\begin{Expl}
Let $X$ be the interval $[0,2]$, and put $x_0 := 0$ and $x_n := 2 - 2^{-n}$ for
all integers $n \ge 1$. The set $A \defl \{x_0,x_1,\dots\}$ spans $X$, but 
$A$ does not strictly span $X$, because $2 \not\in A$.
Let $Y$ be the simplicial metric tree with a single interior vertex~$y_1$ 
and the countably many edges $y_0y_1$ and $y_1y_n$ for $n = 2,3,\dots$, where
$|y_0y_1| = 2^{-1}$ and $|y_1y_n| = 2^{-1} - 2^{-n}$. Note that $Y$ is complete,
hence injective. The map $f \colon A \to Y$ defined by 
$f(x_n) \defl y_n$ for $n = 0,1,2,\dots$ is $1$-roughly isometric, as is 
easily checked. Since there is no pair of points at distance one in $Y$, 
$f$ does not admit a $1$-roughly isometric extension 
$\bar f \colon X \to Y$.
\end{Expl}

However, the following holds. 

\begin{Lem} \label{Lem:dense-subtree}
Let $X$ be a metric tree, and suppose that $A \sub X$ is a set that spans~$X$.
Then there exists a dense subtree $\Sig \sub X$ such that $A \sub \Sig$ and 
$A$ strictly spans~$\Sig$.  
\end{Lem}

\begin{proof}
Let $\Sig$ be the union of all geodesic segments with both
endpoints in $A$. Since $X$ is a metric tree, it is easily seen that 
for every pair of points $x,x' \in \Sig$ the geodesic segment $xx'$ in $X$
is part of a geodesic segment $aa'$ with $a,a' \in A$. In particular, 
$\Sig$ is a geodesic subspace of $X$, hence a metric tree, and $A$ strictly 
spans $\Sig$. It remains to show that $\Sig$ is dense in $X$.
Let $x \in X$. Fix an arbitrary $a \in A$. Since $A$ spans $X$, for every
$\eps > 0$ there is an $a_\eps \in A$ so that $|ax| + |xa_\eps| \le |a a_{\eps}| 
+ \eps$. Consider the geodesic segment $aa_\eps$. Let $x_\eps$ be the 
point on $aa_\eps$ nearest to $x$. Then
\[
2|xx_\eps| = |ax| + |x a_\eps| - |aa_\eps| \le \eps. 
\]
Since $\eps > 0$ was arbitrary and $x_\eps \in \Sig$, 
$x$ lies in the closure of $\Sig$.
\end{proof}


\section{Gromov--Hausdorff distance estimates}

In this section we prove the results stated in the introduction.
First we recall the definition of the Gromov--Hausdorff distance.
Let $(Z,d^Z)$ be a metric space. The usual \emph{Hausdorff distance} 
$d^Z_\H(X,Y)$ of two subsets $X,Y$ of $Z$ is the infimum of all $\rho > 0$ 
such that $X$ is contained in the (open) $\rho$-neighborhood of $Y$ and, 
vice versa, $Y$ lies in the $\rho$-neighborhood of $X$.  
More generally, if $X$ and $Y$ are two metric spaces, their 
\emph{Gromov--Hausdorff distance} $d_\GH(X,Y)$ is defined as the infimum of 
all $\rho > 0$ for which there exist a metric space $(Z,d^Z)$ and isometric 
copies $X',Y' \sub Z$ of $X$ and $Y$, respectively, such that 
$d^Z_H(X',Y') < \rho$. The distance is always finite if $X$ and $Y$ are 
bounded, and for general metric spaces $X_1,X_2,X_3$ the triangle inequality
$d_\GH(X_1,X_2) + d_\GH(X_2,X_3) \ge d_\GH(X_1,X_3)$ holds. Furthermore,
$d_\GH$ induces an honest metric on the set of isometry classes of compact 
metric spaces. 

The Gromov--Hausdorff distance of two metric spaces $X,Y$
may alternatively be characterized as follows.
A \emph{correspondence} $R$ between $X$ and $Y$ is a subset of $X \times Y$ 
such that the projections $\pi_X : X \times Y \to X$ and 
$\pi_Y : X \times Y \to Y$ are surjective when restricted to $R$. 
Then
\[ 
d_\GH(X,Y) = \frac{1}{2} \inf_R \dis(R), 
\]
where the infimum is taken over all correspondences $R$ between $X$ and $Y$
(see~\cite[Theorem 7.3.25]{BurBI}). In view of this characterization, 
the following two theorems are now easy consequences of the results in the 
previous section. 

\begin{Thm}\label{Thm:injective-GH}
Suppose that $X,Y$ are two injective metric spaces, $A \sub X$ is a set
that spans $X$, and $B \sub Y$ is a set that spans $Y$.
Then 
\[
d_\GH(X,Y) \le 2\,d_\GH(A,B).
\]
\end{Thm}

\begin{proof}
Suppose that $R \sub A \times B$ is a correspondence between $A$ and $B$
with $\alpha \defl \dis(R)/2 < \infty$.
By Proposition~\ref{Prop:injective-ext}, there is an extension 
$R \sub R_1 \sub X \times Y$ such that $\pi_X(R_1)$ is an $\alpha$-net in $X$ 
and $\dis(R_1) = \dis(R)$, and there is a further extension 
$R_1 \sub R_2 \sub X \times Y$ such that $\pi_Y(R_2)$ is an $\alpha$-net in $Y$ 
and $\dis(R_2) = \dis(R_1)$. It is then easy to see how to extend $R_2$ 
to a correspondence $\bar R$ between $X$ and $Y$ so that 
$\dis(\bar R) \le \dis(R_2) + 2 \alpha = 2\dis(R)$. Hence,
\[
d_\GH(X,Y) \le \frac12 \dis(\bar R) \le \dis(R),
\] 
and taking the infimum over all correspondences $R$ between $A$ and $B$ 
with finite distortion we obtain the result. 
\end{proof}

\begin{Thm}\label{Thm:trees-GH}
Suppose that $X,Y$ are two metric trees, $A \sub X$ is a set
that spans~$X$, and $B \sub Y$ is a set that spans~$Y$. Then 
\[
d_\GH(X,Y) \le d_\GH(A,B).
\]
\end{Thm}

\begin{proof}
Note that the completions $\bar X, \bar Y$ of $X,Y$ satisfy 
$d_\GH(\bar X,\bar Y) = d_\GH(X,Y)$, and $A,B$ span $\bar X,\bar Y$, 
respectively. We thus assume, without loss of generality, 
that the metric trees $X,Y$ are complete, hence injective.
Let $R \sub A \times B$ be a correspondence between $A$ and~$B$.
By Lemma~\ref{Lem:dense-subtree}, $A$ strictly spans a 
tree $X' \supset A$ that is dense in~$X$. Hence, by 
Proposition~\ref{Prop:trees-ext}, there is an extension 
$R \sub R_1 \sub X' \times Y$ such that $\pi_{X'}(R_1) = X'$
and $\dis(R_1) = \dis(R)$. We have $B \sub B' \defl \pi_Y(R_1)$, and
so $B'$ also spans~$Y$. Again, $B'$ strictly spans
a tree $Y' \supset B'$ that is dense in~$Y$, and there is an extension 
$R_1 \sub R_2 \sub X \times Y'$ such that $\pi_{Y'}(R_2) = Y'$
and $\dis(R_2) = \dis(R_1)$. Since $\pi_X(R_2) \supset X'$ is 
dense in~$X$, and $Y'$ is dense in~$Y$, we obtain that
\[
d_\GH(X,Y) = d_\GH(\pi_X(R_2),Y') \le \frac12 \dis(R_2) = \frac12 \dis(R). 
\] 
As this holds for all correspondences $R$ between $A$ and $B$, 
this gives the result.
\end{proof}

Next, in order to relate these results to the discussion in the introduction, 
we recall Isbell's explicit construction of the injective hull 
$\E X$ of a metric space~$X$. We denote by $\R^X$ the vector space of all real 
functions on~$X$. As a set, $\E X$ is defined as
\[
\E X \defl \bigl\{ f \in \R^X : 
\text{$f(x) = \textstyle\sup_{y \in X} (|xy| - f(y))$ for all $x \in X$} 
\bigr\},
\]
the set of the so-called {\em extremal functions} on $X$. For every $z \in X$,
the distance function $d_z$, defined by $d_z(x) \defl |xz|$ for $x \in X$,
belongs to~$\E X$. In general, for every $f \in \E X$ and $z \in X$, the 
inequalities 
\[   
d_z - f(z) \le f \le d_z + f(z)
\]
hold, and it follows that $\|f - d_z\|_\infty \defl \sup |f - d_z|
= f(z)$. In particular, $\|f - g\|_\infty$ is finite for every pair of 
functions $f,g \in \E X$, and this equips $\E X$ with a metric.
The map $\e \colon X \to \E X$ that takes $x$ to $d_x$ is then a canonical 
isometric embedding of $X$ into $\E X$, as $\|d_x - d_y\|_\infty = |xy|$ for 
all $x,y \in X$. Isbell proved that $(\e,\E X)$ is indeed an 
\emph{injective hull} of $X$, i.e., $\E X$ is an injective metric space, 
and $(\e,\E X)$ is a minimal such extension of $X$ in that no proper subspace 
of $\E X$ containing $\e(X)$ is injective. Furthermore, if $(i,Y)$ is 
another injective hull of $X$, then there exists a unique isometry
$I \colon \E X \to Y$ with the property that $I \circ \e = i$.
The following result explains how injective 
hulls are related to spanning subsets of (injective) metric spaces,
in the sense of this paper. 

\begin{Prop} \label{Prop:inj-hull-spanned}
\begin{enumerate}
\item 
For every metric space $A$, the image $\e(A)$ of the canonical 
isometric embedding $\e \colon A \to \E A$ spans~$\E A$.
\item 
If $X$ is an injective metric space and $A \sub X$ is a set that spans
$X$, then $X$ is isometric to $\E A$ via the map that 
sends $x \in X$ to the restricted distance function~$d_x|_A$. 
\end{enumerate}
\end{Prop}

\begin{proof}
For~(1), let a pair $(f,g)$ of elements of $\E A$ be given, and let 
$\eps > 0$. There exists either a point 
$b \in A$ such that $\|f - g\|_\infty \le f(b) - g(b) + \eps/2$ or a point 
$a \in A$ such that $\|f - g\|_\infty \le g(a) - f(a) + \eps/2$. Then, by the 
definition of $\E A$, we may choose $a \in A$ with 
$f(b) \le |ab| - f(a) + \eps/2$ in the first case and $b \in A$ with 
$g(a) \le |ab| - g(b) + \eps/2$ in the second. In either case, this gives
\[
\|f - g\|_\infty \le |ab| - f(a) - g(b) + \eps.
\]
Since $|ab| - f(a) \le f(b) = \|f - d_b\|_\infty$ and 
$g(b) = \|g - d_b\|_\infty$, we obtain that 
$\|f - g\|_\infty \le \|f - d_b\|_\infty - \|g - d_b\|_\infty + \eps$.
As $d_b = \e(b) \in \e(A)$, this shows the claim.

For the proof of~(2), let $x,y \in X$. Since $A$ spans $X$, we have first 
that for every $a \in A$, $d_x(a) = \sup_{b \in A}(|ab| - d_x(b))$, 
so $d_x|_A \in \E A$.
Secondly, $|xy| = \sup_{a \in A}(|ax| - |ay|)$, which implies that 
the inequality
\[
\bigl\| d_x|_A - d_y|_A \bigr\|_\infty 
= \sup_{a \in A}\bigl| |ax| - |ay| \bigr| \le |xy|
\]
is in fact an equality. Hence, the map that takes $x$ to $d_x|_A$ is an 
isometric embedding of $X$ into $\E A$. Since $X$ is injective, so is the 
image of this map. Because no proper subspace of $\E A$ containing 
$\e(A)$ is injective, the image agrees with $\E A$.
\end{proof}

In view of Proposition~\ref{Prop:inj-hull-spanned}, 
Theorem~\ref{Thm:injective-GH} is equivalent to saying that for any 
metric spaces $A$ and $B$,
\[
d_\GH(\E A, \E B) \le 2\,d_\GH(A,B),
\]
as stated in the introduction.
We now show that the factor two is optimal. 

\begin{Expl}\label{Expl:roughlyiso}
First we show that 
if $f \colon \R \times [0,4] \rightarrow \R$ is an $\eps$-roughly isometric 
map, where $\R \times [0,4] \sub \R^2$ is endowed with the $l_1$ metric, 
then $\eps \ge 4$. 
For any integer $n \ge 1$, consider the subset 
\[
Z_n \defl \bigl( \{0,8,\dots,8n\} \times \{0\} \bigr) \cup 
\bigl( \{4,12,\dots,8n-4\} \times \{4\} \bigr)
\]
of $\R \times [0,4]$ of cardinality $2n+1$. Note that, with respect to the 
$l_1$ distance, distinct points in $Z_n$ are at distance at least eight 
from each other, and the diameter of $Z_n$ equals~$8n$. 
Let $\{z_1,z_2,\dots,z_{2n+1}\}$ be an enumeration
of $Z_n$ so that $f(z_1) \le f(z_2) \le \dots \le f(z_{2n+1})$. We have 
$f(z_{i+1}) - f(z_i) \ge \|z_{i+1} - z_i\|_1 - \eps \ge 8 - \eps$,
hence taking the sum from $i = 1$ to $2n$ we obtain $f(z_{2n+1}) - f(z_1) \ge
2n(8 - \eps)$. On the other hand, $f(z_{2n+1}) - f(z_1) \le 
\diam(Z_n) + \eps = 8n + \eps$.
It follows that $\eps \ge 8n/(2n+1)$. 
This holds for any $n \ge 1$, thus $\eps \ge 4$. 

Now, for any $N > 0$, consider the two metric spaces 
$A = \{a_1,\dots,a_4\}$ and $B = \{b_1,\dots,b_4\}$, where 
$|a_1a_2| = |a_3a_4| = 4$, $|a_1a_3| = |a_2a_4| = N$, 
$|a_1a_4| = |a_2a_3| = N+4$, $|b_1b_2| = |b_3b_4| = 2$ and 
$|b_ib_j| = N+2$ ($i \ne j$) otherwise. The correspondence 
$\{(a_1,b_1),\dots,(a_4,b_4)\}$ has distortion~two, and since
$\diam(A) = \diam(B) + 2$ there is no correspondence 
between~$A$ and~$B$ with distortion less than~two. 
So $d_\GH(A,B) = 1$.
The injective hull $\E A$ is isometric to $[0,N] \times [0,4] \sub \R^2$ with
the $l_1$ distance, and $\E B$ is a metric tree with a central edge of 
length~$N$ and two edges of length~one attached at each of its endpoints
(like the tree $Y$ depicted in the introduction).
Let $\eps_0 < 4$ be given. If $N$ is chosen big enough, depending on $\eps_0$,
essentially the same argument as above shows that there is no 
$\eps$-roughly isometric map $f \colon \E A \to \E B$ with $\eps < \eps_0$. 
In particular, every correspondence between $\E A$ and $\E B$ has distortion 
at least $\eps_0/2$. In other words, for every $\del_0 < 2$ we find a pair of 
four-point metric spaces $A,B$ so that $\E A$ is two-dimensional, 
$\E B$ is a metric tree, $d_\GH(A,B) = 1$, and $d_\GH(\E A,\E B) \ge \del_0$.
\end{Expl}


\end{document}